\documentclass{amsart}
\usepackage{ucs}
\usepackage[utf8x]{inputenc}
\usepackage{verbatim}
\usepackage{paralist}
\usepackage{leftidx}
\usepackage{dsfont}
\usepackage{amsthm}
\usepackage{amsmath}
\usepackage{amssymb}
\usepackage{amscd}
\usepackage{graphicx}
\usepackage{setspace}
\usepackage[all]{xy}
\usepackage{mathrsfs} 
\usepackage{hyperref}
\usepackage{stmaryrd}
\title[Degeneration of Riemann theta functions]{Degeneration of Riemann theta functions and of the Zhang--Kawazumi invariant with applications to a uniform Bogomolov conjecture}
\author{Robert Wilms}
\thanks{The author gratefully acknowledges support from the Swiss National Science Foundation grant ``Diophantine Equations: Special Points, Integrality, and Beyond'' (n$^\circ$ 200020\_184623)}
\date{\today}
\begin{document}
	\numberwithin{equation}{section}
	\newtheorem{Def}{Definition}
	\numberwithin{Def}{section}
	\newtheorem{Rem}[Def]{Remark}
	\newtheorem{Lem}[Def]{Lemma}
	\newtheorem{Que}[Def]{Question}
	\newtheorem{Cor}[Def]{Corollary}
	\newtheorem{Exam}[Def]{Example}
	\newtheorem{Thm}[Def]{Theorem}
	\newtheorem*{clm}{Claim}
	\newtheorem{Pro}[Def]{Proposition}
	\newcommand\gf[2]{\genfrac{}{}{0pt}{}{#1}{#2}}
	
\begin{abstract}
In this paper we study the degeneration behavior of the norm of the Riemann $\theta$-function in a family of principally polarized abelian varieties over the punctured complex unit disc in terms of the associated polarized real torus. As an application, we obtain the degeneration behavior of the Zhang--Kawazumi invariant $\varphi(M_t)$ of a family of Riemann surfaces $M_t$ in terms of Zhang's invariant $\varphi(\Gamma)$ of the associated metrized reduction graph $\Gamma$. This allows us to deduce a uniform lower bound for the essential minimum of the Néron--Tate height on the tautological cycles of any Jacobian variety over a number field.
\end{abstract}
\maketitle
\section{Introduction}
In the intersection theory of arithmetic surfaces, introduced by Arakelov in \cite{Ara74}, there are various invariants associated to Riemann surfaces playing a crucial role. In this article we would like to study their degeneration behavior in terms of the tropical geometry associated to the degeneration data. This has been done by de Jong \cite{dJo19} for Faltings' $\delta$-invariant, an archimedean analogue of the number of singular points on a closed fiber introduced by Faltings in \cite{Fal84}, and for the Arakelov--Green function, which defines the archimedean contribution for the intersection pairing.

Here, we will focus on the degeneration of the Riemann $\theta$-function. Although we are especially interested in the case of Jacobian varieties, we will study it for general abelian varieties.
Let $(A,\Theta)$ be any principally polarized complex abelian variety of dimension $g\ge 1$, where $\Theta\subseteq A$ denotes a divisor, such that $\mathcal{O}_A(\Theta)$ is an ample and symmetric line bundle satisfying $\dim H^0(A,\mathcal{O}_A(\Theta))=1$ and inducing the principle polarization. There exists a unique $C^{\infty}$-function $\|\theta\|^2\colon A\to\mathbb{R}_{\ge0}$ satisfying
$$\partial\overline{\partial}\log\|\theta\|^2=2\pi i(\nu-\delta_\Theta) \quad \text{and}\quad \int_A\|\theta\|^2\frac{\nu^g}{g!}=2^{-g/2},$$
where $\nu$ denotes the canonical 1-1 form associated to $(A,\Theta)$. Further, there exists a matrix $\tau\in \mathbb{H}_g$, such that $A\cong \mathbb{C}^g/(\mathbb{Z}^g+\tau \mathbb{Z}^g)$ and with the Riemann $\theta$-function $\theta(\tau,z)=\sum_{n\in\mathbb{Z}^g}\exp(\pi i\ltrans{n}\tau n+ 2\pi i\ltrans{n}z)$ for $z\in \mathbb{C}^g$ we can write 
$$\|\theta\|^2(z ~\mathrm{mod}~ \mathbb{Z}^g+\tau \mathbb{Z}^g)=\det(\mathrm{Im}~\tau)^{1/2}\exp(-2\pi \ltrans{(\mathrm{Im}~z)}(\mathrm{Im}~\tau)^{-1}(\mathrm{Im}~z))|\theta|(\tau,z)^2.$$
In our main theorem we would like to describe the degeneration behavior of the invariant
$$I(A,\Theta)=\log\int_A\|\theta\|^2(z)\frac{\nu^g(z)}{g!}-\int_A\log\|\theta\|^2(z)\frac{\nu^g(z)}{g!}.$$

For this purpose, we consider a smooth family $f\colon\mathscr{A}\to\Delta^*$ of principally polarized complex abelian varieties over the punctured unit disk $\Delta^*=\{z\in\mathbb{C}^*~|~|z|<1\}$ and we denote $\Sigma_f$ for its associated polarized real torus, see Sections \ref{sec_tropical} and \ref{sec_families} for its definition.
There also exists a tropical Riemann theta function $\|\Psi\|\colon \Sigma_f\to\mathbb{R}$. The tropical moment is given by $I(\Sigma_f)=2\int_{\Sigma_f}\|\Psi\|\mu_H$ with $\mu_H$ the Haar measure of volume $1$ on $\Sigma_f$. We write $f_1\sim f_2$ for two continuous functions $f_1,f_2\colon \Delta^*\to \mathbb{R}$ if $f_1-f_2$ can be extended to a continuous function on $\Delta=\Delta^*\cup \{0\}$.
\begin{Thm}\label{mainthm}
	Let $f\colon \mathscr{A}\to\Delta^*$ be any smooth family of principally polarized complex abelian varieties. It holds
	$$I(\mathscr{A}_t)\sim -I(\Sigma_f)\log|t|-\frac{\dim \Sigma_f}{2}\log(-\log|t|).$$
\end{Thm}

For the proof we first describe the degeneration behavior of $\|\theta\|$ in terms of the tropical Riemann theta function $\|\Psi\|$ by an explicit computation. Then the theorem follows by an integration. To avoid that the integration domain depends on $t$, we integrate over $[0,1)^{2g}$, which is isomorphic to $\mathcal{A}_t$ as a measure space, where both are equipped with their Haar measure.

The invariants $I(A,\Theta)$ and $I(\Sigma_f)$ play a crucial role in Arakelov theory. For example, de Jong and Shokrieh \cite{dJS18a} have recently proved, that these are the local contributions for a comparison between the Faltings height of an abelian variety and the Néron--Tate height of its theta divisor. As another application we proved in \cite{Wil17}, that $I(A,\Theta)$ can be expressed as a linear combination of Faltings' $\delta$-invariant and the Zhang--Kawazumi invariant $\varphi$, introduced independently by Zhang \cite{Zha10} and Kawazumi \cite{Kaw08}, in the case of Jacobians of Riemann surfaces. Explicitly, we have 
$$\delta(M)=12 I(\mathrm{Jac}(M))+2\varphi(M)-2g\log 2\pi^4$$
for any compact and connected Riemann surface $M$ of genus $g\ge 1$. The following variant of this formula for the tropical invariants has been proved by de Jong and Shokrieh \cite{dJS18b}
$$\delta(\Gamma)+\epsilon(\Gamma)=12 I(\mathrm{Jac}(\Gamma))+2\varphi(\Gamma),$$
where $\Gamma$ denotes any polarized metrized graph, $\delta(\Gamma)$ is the length of $\Gamma$ and $\epsilon(\Gamma)$ and $\varphi(\Gamma)$ are invariants of $\Gamma$ introduced by Zhang in \cite{Zha93} respectively \cite{Zha10}. We will explain the terms in this formula in more details in Section \ref{sec_graphs}.

The result of Theorem \ref{mainthm} is very similar to the formula for the degeneration behavior of Faltings' $\delta$-invariant proved by de Jong \cite{dJo19}. To recall his result, let $\pi\colon \mathcal{X}\to\Delta$ be any smooth family of semistable complex projective curves having smooth fibers over $\Delta^*$ and $\Gamma_\pi$ its associated polarized metrized reduction graph. Then it holds
\begin{align}\label{equ_delta-dejong-intro}
\delta({\mathcal{X}_t})\sim -(\delta(\Gamma_\pi)+\epsilon(\Gamma_\pi))\log|t|-6 g_0(\Gamma_\pi)\log(-\log|t|).
\end{align}
Here, $g_0(\Gamma_\pi)$ denotes the genus of the underlying graph of $\Gamma_\pi$. The expression of the last term differs from de Jong's expression and will be justified in Section \ref{sec_families_riemann}.
In particular, we can combine both formulas to obtain the degeneration behavior of the Zhang--Kawazumi invariant.
\begin{Cor}\label{cor_ZK}
	Let $f\colon \mathcal{X}\to \Delta$ be a smooth family of semistable complex projective curves of genus $g\ge 1$, where any fiber above $\Delta^*$ is smooth. Then it holds
	$$\varphi(\mathcal{X}_t)\sim -\varphi(\Gamma_\pi)\log|t|.$$
\end{Cor}
This generalizes the result obtained by de Jong \cite{dJo14} about the asymptotic behavior of $\varphi$ in the case where the special fiber $\mathcal{X}_0$ has exactly one node.
By the lower bound of $\varphi(\Gamma_\pi)$ in terms of $\delta(\Gamma_\pi)$ by Cinkir \cite{Cin11} and by a compactness argument on the moduli space of Riemann surfaces we will obtain the following corollary.
\begin{Cor}\label{cor_lowerbound}
	For any integer $g\ge 2$ there exist constants $c_1(g),c_2(g)\in\mathbb{R}$ with $c_2(g)>0$ such that any compact and connected Riemann surface $M$ of genus $g$ satisfies
	$$\varphi(M)\ge \tfrac{(g-1)^2}{23g^2+11g+2}\max(\delta(M)+c_1(g),c_2(g)).$$
\end{Cor}
The invariants $\varphi$ for Riemann surfaces and for polarized metrized graphs also have important applications in Arakelov theory. To explain one of these, let $X$ be any smooth projective geometrically connected curve of genus $g\ge 2$ over a number field $K$ of degree $d_K$ with semistable reduction over the spectrum of the integers $B=\mathrm{Spec}(\mathcal{O}_K)$ and denote $\pi\colon \mathcal{X}\to B$ for its minimal regular model. For a maximal prime ideal $v\in |B|$ we denote $\Gamma_v(X)$ for the polarized metrized reduction graph of $\mathcal{X}$ at $v$. Similarly, for any embedding $v\colon K\to \mathbb{C}$ we write $X_v$ for the Riemann surface obtained by the base change induced by $v$. We set
$$\varphi(X)=\sum_{v\in|B|}\varphi(\Gamma(\mathcal{X}_v))\log N(v)+\sum_{v\colon K\to\mathbb{C}}\varphi(X_v)$$
as the weighted sum of the local contributions by the invariants $\varphi$.
Zhang \cite{Zha10} has proved, that the arithmetic self-intersection number of the Gross--Schoen cycle $\Delta_{\xi}\subseteq X^3$, which is a modified version of the diagonal, can be expressed by the arithmetic self-intersection number of the canonical bundle $\hat{\omega}_X$ equipped with its canonical admissible adelic metric and the invariant $\varphi$
$$\tfrac{2g+1}{2g-2}\hat{\omega}_X^2=\langle \Delta_{\xi},\Delta_{\xi}\rangle+\varphi(X).$$
He also remarked that the arithmetic analogue of Grothendieck's Standard conjectures by Gillet--Soulé \cite{GS94} predicts, that it should hold $\langle \Delta_{\xi},\Delta_{\xi}\rangle\ge 0$, which would give the best possible lower bound for $\hat{\omega}_X^2$ in terms of $\varphi(X)$, as $\langle \Delta_{\xi},\Delta_{\xi}\rangle$ is $0$ for hyperelliptic curves. But this is unfortunately not known to be true unconditionally. 

Instead, in \cite{Wil19} we found the following weaker bound for $\hat{\omega}_X^2$ in terms of $\varphi(X)$
$$\tfrac{2g+1}{g-1}\hat{\omega}_X^2\ge \varphi(X).$$
Combining this bound with Corollary \ref{cor_lowerbound}, the arithmetic Noether formula and Cinkir's lower bound, we obtain the following lower bound of $\hat{\omega}^2_X$ in terms of the Faltings height $h_{\mathrm{Fal}}(X)$.

\begin{Cor}\label{cor_height-bound}
	Let $g\ge 2$ be any integer and $c_1(g), c_2(g)$ as in Corollary \ref{cor_lowerbound}. Any smooth projective geometrically connected curve $X$ of genus $g\ge 2$ defined over a number field $K$ and with semi-stable reduction over $\mathrm{Spec}~\mathcal{O}_K$ satisfies 
	$$\hat{\omega}_X^2\ge \tfrac{d_K(g-1)^3}{47g^3+42g^2+18g+1}\max(12h_{\mathrm{Fal}}(X)+c_1(g),c_2(g)).$$
\end{Cor}

Such bounds have applications to the Bogomolov conjecture. Let $A$ be any principally polarized abelian variety over $K$ and $\mathcal{L}$ a symmetric ample line bundle on $A$ inducing the principal polarization on $A$. For any closed positive dimensional subvariety $Z\subseteq A$ the essential minimum is defined by
$$e_{\mathcal{L}}'(Z)=\sup_{\gf{Y\subseteq Z}{\mathrm{codim}(Y)=1}} \inf_{x\in(Z\setminus Y)(\overline{K})} h'_{\mathcal{L}}(x),$$
where $h'_{\mathcal{L}}$ denotes the Néron--Tate height associated to $\mathcal{L}$. The Bogomolov conjecture states that $e'_{\mathcal{L}}(Z)>0$ if $Z$ is not the translate by a torsion point of an abelian subvariety of $A$. That means, that there exists an $\epsilon>0$, such that the set of geometric points $x\in Z(\overline{K})$ satisfying $h'_{\mathcal{L}}(x)<\epsilon$ is not Zariski dense in $Z$. In fact every $0<\epsilon<e'_{\mathcal{L}}$ does this job. This was first proven by Ullmo \cite{Ull98} if $Z$ is a curve embedded in its Jacobian $A=\mathrm{Pic}^0(Z)$ and then by Zhang \cite{Zha98} in the general case.

We define the Néron--Tate height of $Z$ associated to $\mathcal{L}$ to be 
$$h'_{\mathcal{L}}(Z)=\frac{\langle \hat{\mathcal{L}}^{\dim Z+1}|Z\rangle}{d_K(\dim Z+1)\langle\mathcal{L}^{\dim Z}|Z\rangle},$$
where the brackets denote the (arithmetic) self-intersection numbers of $\mathcal{L}$ equipped with its admissible adelic metric and restricted to $Z$. Due to Zhang \cite[Theorem~1.10]{Zha95}, we know that $e'_{\mathcal{L}}(Z)\ge h'_{\mathcal{L}}(Z)$.

Let now $A=\mathrm{Pic}^0(X)$ be the Jacobian of the curve above, $\alpha\in\mathrm{Div}^1(X)$ a divisor of degree $\deg\alpha=1$, $m\in (\mathbb{Z}\setminus\{0\})^r$ an $r$-dimensional vector of non-zero integers and $Z=Z_{m,\alpha}$ the image of the map
$$f_{m,\alpha}\colon X^r\to A,\quad (x_1,\dots,x_r)\mapsto \sum_{j=1}^r m_j(x_j-\alpha).$$
In \cite{Wil19} we computed $h'_{\mathcal{L}}(Z_{m,\alpha})$ in terms of $\hat{\omega}_X^2$, $\varphi(X)$ and $h'_{\mathcal{L}}\left(\alpha-\tfrac{1}{2g-2}\omega_X\right)$.
Combining this result with the bound in Corollary \ref{cor_height-bound}, we are able to deduce the following uniform bound for the Bogomolov conjecture in this case.
\begin{Cor}\label{cor_uniform-bogomolov}
	Let $X$ be any smooth projective geometrically connected curve $X$ of genus $g\ge 2$ defined over a number field $K$ and with semi-stable reduction over $\mathrm{Spec}~\mathcal{O}_K$, $\alpha\in\mathrm{Div}^1(X)$ any degree $1$ divisor and $m\in(\mathbb{Z}\setminus\{0\})^r$ any vector of $r\in\{1,\dots,g-1\}$ non-zero integers. Then it holds
	$$e'_{\mathcal{L}}(Z_{m,\alpha})\ge h_{\mathcal{L}}'(Z_{m,\alpha})\ge\tfrac{(g-1)^3}{24(47g^4+42g^3+18g^2+g)}\max(12h_{\mathrm{Fal}}(X)+c_1(g),c_2(g))$$
	with $c_1(g)$ and $c_2(g)$ as in Corollary \ref{cor_lowerbound}.
\end{Cor}
If we split this lower bound into the two lower bounds
\begin{align}\label{2lowerbounds}
e'_{\mathcal{L}}\ge\tfrac{(g-1)^3}{24(47g^4+42g^3+18g^2+g)}(12h_{\mathrm{Fal}}(X)+c_1(g)), \quad e'_{\mathcal{L}}\ge\tfrac{(g-1)^3}{24(47g^4+42g^3+18g^2+g)}c_2(g),
\end{align}
we may interpret the corollary as a uniform version of Bogomolov's conjecture in two ways. The first one gives a lower bound of the essential minimum increasing with the Faltings height and the second one gives a uniform positive lower bound of the essential minimum only depending on the genus $g$.

The first bound can be compared with the work by Dimitrov, Gao and Habegger in \cite[Theorem 1.6]{DGH20a}, where they found a similar bound for the essential minimum for an irreducible non-degenerate and dominating subvariety of a family of abelian varieties. Their result is much more general in the sense, that it treats not only tautological cycles in a Jacobian and it concerns the essential minimum of the entire family and not the fiberwise essential minima. On the other hand, it is not clear how to obtain a bound for the Néron--Tate height $h'_{\mathcal{L}}$ and how to obtain the second lower bound in (\ref{2lowerbounds}) from their work. For the special case of the image of the differential morphism of the $(M+1)$-th power of the curve into the $M$-th power of its Jacobian, they conditionally deduce a bound similar to the second bound in (\ref{2lowerbounds}) from a relative Bogomolov conjecture in \cite[Proposition 2.3]{DGH20b}.

\subsubsection*{Outline}
The first part of this paper deals with the proof of Theorem \ref{mainthm}. In Section \ref{sec_tropical}~ we recall the notion of a polarized real torus and its associated Riemann $\theta$-function. We consider degenerating families of abelian varieties and their associated polarized real tori in Section \ref{sec_families}. By explicit computations we study the degeneration behavior of the function $\|\theta\|$ in Section \ref{sec_theta}. This allows us to prove Theorem \ref{mainthm} in the subsequent section.

In the second part of this paper we are interested in applications of Theorem \ref{mainthm} to curves, first to complex curves and later to the arithmetic situation of curves over number fields. In Section \ref{sec_graphs} we recall the required facts about metrized graphs. There is nothing original in this section. We study families of Riemann surfaces in Section \ref{sec_families_riemann}. In particular, we prove that the tropical Jacobian of the metrized graph associated to a degenerating family of Riemann surfaces is isomorphic to the polarized real torus associated to the corresponding family of Jacobians. Moreover, we prove Corollaries \ref{cor_ZK} and \ref{cor_lowerbound} in this section. Finally in Section \ref{sec_arithmetic}, we consider the arithmetic situation of smooth projective curves defined over number fields. In particular, we prove Corollaries \ref{cor_height-bound} and \ref{cor_uniform-bogomolov} there.

\section{Polarized real tori}\label{sec_tropical}
In this section we recall the notion of a polarized real torus and its associated tropical Riemann theta function and tropical moment. We refer to \cite{FRSS18} and \cite[Section 2]{dJS18b} for details.

An \emph{Euclidean lattice} is a pair $(\Lambda,[\cdot,\cdot])$, where $\Lambda$ is a finitely generated free $\mathbb{Z}$-module and 
$$[\cdot,\cdot]\colon \Lambda_\mathbb{R}\times\Lambda_\mathbb{R}\to \mathbb{R}$$
is an inner product on $\Lambda_\mathbb{R}=\Lambda\otimes_\mathbb{Z}\mathbb{R}$. A \emph{polarized real torus} is the compact Riemannian manifold $\Sigma=\Lambda_\mathbb{R}/\Lambda$ associated to an Euclidean lattice $(\Lambda,[\cdot,\cdot])$.

The tropical Riemann theta function associated to $\Sigma$ is defined by
$$\Psi\colon \Lambda_\mathbb{R}\to \mathbb{R},\quad \Psi(\nu)=\min_{\lambda\in \Lambda}\left\{\tfrac{1}{2}[\lambda,\lambda]+[\nu,\lambda]\right\}.$$
This can be modified to an $\Lambda$-invariant function descending to $\Sigma$
$$\|\Psi\|\colon \Sigma\to \mathbb{R}, \quad \|\Psi\|(\nu)=\Psi(\nu)+\tfrac{1}{2}[\nu,\nu]=\min_{\lambda\in \Lambda}\tfrac{1}{2}[\nu+\lambda,\nu+\lambda].$$
The tropical moment of $\Sigma$ is given by 
$$I(\Sigma)=2\int_\Sigma\|\Psi\|\mu_H,$$
where $\mu_H$ denotes the Haar measure on $\Sigma$ with volume $1$.

Let $g\ge 1$ be an integer. To any positive-definite symmetric matrix $B\in\mathbb{R}^{g\times g}$ we associate an Euclidean lattice $B\mathbb{Z}^g=(B\mathbb{Z}^g,[\cdot,\cdot]_B)$, where the inner product is induced by $[Bn, Bn']_B=\ltrans{n}Bn'$ for any $n, n'\in\mathbb{Z}^g$. We denote the associated polarized real torus by $\Sigma_B=\mathbb{R}^g/B\mathbb{Z}^g$, where we obtain $[x,y]_B=\ltrans{x}B^{-1}y$ for $x,y\in \mathbb{R}^g$.
For $x\in \mathbb{R}^g$ we obtain for the modified version of the theta function $\Psi_B$ associated to $\Sigma_B$
\begin{align}\label{equ_valuetheta}
\|\Psi_B\|(Bx)=\min_{Bn\in B\mathbb{Z}^g}\tfrac{1}{2}[Bx+Bn,Bx+Bn]_B=\tfrac{1}{2}\min_{n\in \mathbb{Z}^g}\ltrans{(x+n)}B(x+n).
\end{align}
Further, let us remark that the tropical moment of $\Sigma_B$ can be expressed by 
\begin{align}\label{equ_tropmoment}
I(\Sigma_B)=2\int_{x\in [0,1]^{g}}\|\Psi\|(Bx)\lambda(x),
\end{align}
where $\lambda$ denotes the Lebesgue measure on $\mathbb{R}^{g}$.
\section{Families of abelian varieties}\label{sec_families}
We discuss families of abelian varieties over the punctured unit disc and their associated polarized real tori in this section.
Let $\Delta=\{z\in \mathbb{C}~|~|z|<1\}$ be the complex open unit disc and $\Delta^*=\Delta\setminus \{0\}$. Further, let $f\colon\mathscr{A}\to \Delta^*$ be a smooth family of principally polarized complex abelian varieties of dimension $g\ge 1$. We denote by $\mathbb{H}_g=\{\tau\in\mathbb{C}^{g\times g}~|~\tau=\ltrans{\tau},~\mathrm{Im}~\tau>0\}$ the Siegel upper half-space. According to \cite[§17.]{Nam76}, there exists a corresponding multi-valued holomorphic period map $T_f\colon \Delta^*\to \mathbb{H}_g$ such that for some positive integer $m$ we have
\begin{align}\label{equ_period-matrix}
T_f(t^m)=\begin{pmatrix}
S_1(t)&S_3(t)\\
\ltrans{S_3(t)}&\frac{m\log t}{2\pi i}B + S_2(t)
\end{pmatrix},
\end{align}
where $S_1\colon \Delta\to \mathbb{H}_{g_1}$, $S_3\colon \Delta\to \mathbb{C}^{g_1\times g_2}$ and $S_2\colon \Delta\to \mathbb{C}^{g_2\times g_2}$ are single-valued holomorphic maps on the whole unit disc $\Delta$ for some integers $g_1+g_2=g$ and $B\in \mathbb{Q}^{g_2\times g_2}$ is a symmetric and positive-definite matrix. We define the polarized real torus associated to $f$ as $\Sigma_f=\Sigma_{B}$. In particular, it holds $g_2=\dim \Sigma_f$.

\begin{Lem}
	The polarized real torus $\Sigma_f$ of a smooth family of principally polarized complex abelian varieties $f\colon \mathcal{A}\to \Delta^*$ does not depend on the choice of the period map $T_f\colon \Delta\to \mathbb{H}_g$ up to isomorphism.
\end{Lem}
\begin{proof}
	Two choices $T_f, T'_f\colon \Delta^*\to \mathbb{H}_g$ of a multi-valued holomorphic period map differ by a constant coordinate change of the lattice $\Lambda=H_1(\mathcal{X}_t,\mathbb{Z})$, which respects the polarization. That means, there exists an isomorphism of lattices
	$$\alpha\colon \mathbb{Z}^g+T_f(t)\mathbb{Z}^g\to \mathbb{Z}^g+T'_f(t)\mathbb{Z}^g,$$
	such that 
	\begin{align*}
	&\ltrans{\alpha(\begin{pmatrix} \mathrm{Id}_g & T_f(t)\end{pmatrix}a)}(\mathrm{Im}~T'_f(t))^{-1}\alpha(\begin{pmatrix} \mathrm{Id}_g & T_f(t)\end{pmatrix}b)\\
	&=\ltrans{(\begin{pmatrix} \mathrm{Id}_g & T_f(t)\end{pmatrix}a)}(\mathrm{Im}~T_f(t))^{-1}\begin{pmatrix} \mathrm{Id}_g & T_f(t)\end{pmatrix}b
	\end{align*}
	for any $a,b\in\mathbb{Z}^{2g}$
	and the isomorphism can be represented by a constant matrix $R\in\mathrm{GL}_{2g}(\mathbb{Z})$ such that 
	$$\alpha\left(\begin{pmatrix} \mathrm{Id}_g & T_f(t)\end{pmatrix}a \right)=\begin{pmatrix} \mathrm{Id}_g & T'_f(t)\end{pmatrix} R a.$$
	
	Now let $T_f$ and $T'_f$ be both of the form as in Equation (\ref{equ_period-matrix}). By taking a common multiple, we may assume that the corresponding $m$'s are the same. Concretely, we write
	\begin{align*}
	T_f(t^m)=\begin{pmatrix}
	S_1(t)&S_3(t)\\
	\ltrans{S_3(t)}&\frac{m\log t}{2\pi i}B + S_2(t)
	\end{pmatrix},\quad T'_f(t^m)=\begin{pmatrix}
	S'_1(t)&S'_3(t)\\
	\ltrans{S'_3(t)}&\frac{m\log t}{2\pi i}B' + S'_2(t)
	\end{pmatrix}
	\end{align*}
	As $R$ is constant, it has to respect this form. Hence, we have $R=\begin{pmatrix}R_1&0\\ 0&R_2\end{pmatrix}$ for some $R_2\in\mathrm{GL}_{g_2}(\mathbb{Z})$.
	For $b_1,b_2\in\mathbb{Z}^{g_2}$ we compute
	\begin{align*}
	&\frac{m\arg t}{\pi}\ltrans{(R_2b_1)}B' R_2b_2+2\ltrans{(R_2b_1)}(\mathrm{Re}~S'_2(t)) R_2b_2\\
	=&2\ltrans{\begin{pmatrix}0\\ R_2b_1\end{pmatrix}}(\mathrm{Re}~T'_f(t^m)) \begin{pmatrix}0\\ R_2b_2\end{pmatrix}\\
	=&\mathrm{Im}~\left(\ltrans{\left(T'_f(t^m)\begin{pmatrix}0\\ R_2b_1\end{pmatrix} \right)}(\mathrm{Im}~T'_f(t^m))^{-1} T'_f(t^m)\begin{pmatrix}0\\ R_2b_2\end{pmatrix}\right)\\
	=&\mathrm{Im}~\left(\ltrans{\left(\begin{pmatrix} \mathrm{Id}_g & T'_f(t^m)\end{pmatrix}R\begin{pmatrix} 0\\ b_1\end{pmatrix}\right)}(\mathrm{Im}~T'_f(t^m))^{-1}\begin{pmatrix} \mathrm{Id}_g & T'_f(t^m)\end{pmatrix}R\begin{pmatrix} 0\\ b_2\end{pmatrix}\right)\\
	=&\mathrm{Im}~\left(\ltrans{\alpha\left(\begin{pmatrix} \mathrm{Id}_g & T_f(t^m)\end{pmatrix}\begin{pmatrix} 0\\ b_1\end{pmatrix}\right)}(\mathrm{Im}~T'_f(t^m))^{-1}\alpha\left(\begin{pmatrix} \mathrm{Id}_g & T_f(t^m)\end{pmatrix}\begin{pmatrix} 0\\ b_2\end{pmatrix}\right)\right)\\
	=&\mathrm{Im}~\left(\ltrans{\left(\begin{pmatrix} \mathrm{Id}_g & T_f(t^m)\end{pmatrix}\begin{pmatrix} 0\\ b_1\end{pmatrix}\right)}(\mathrm{Im}~T_f(t^m))^{-1}\left(\begin{pmatrix} \mathrm{Id}_g & T_f(t^m)\end{pmatrix}\begin{pmatrix} 0\\ b_2\end{pmatrix}\right)\right)\\
	=&\frac{m\arg t}{\pi}\ltrans{b_1}B b_2+2\ltrans{b_1}(\mathrm{Re}~S_2(t)) b_2,
	\end{align*}
	where the last equality follows analogously to the first three equations. Comparing the multi-valued parts, we obtain $\ltrans{(R_2b_1)}B' R_2b_2=\ltrans{b_1}B b_2$. Thus, the map
	$$B\mathbb{Z}^{g_2}\to B'\mathbb{Z}^{g_2}, \quad Bb\to B'R_2b$$
	is an isomorphism of Euclidean lattices, where we define the inner product associated to $B\mathbb{Z}^{g_2}$ and $B'\mathbb{Z}^{g_2}$ as in Section \ref{sec_tropical}. Hence, we obtain an isomorphism of the corresponding polarized real tori. This proves the lemma.
	
\end{proof}

For any section $z\colon \Delta^*\to \mathscr{A}$ of $f$ we define
$$\mathrm{trop}(z)\equiv-\lim_{t\to 0}\frac{2\pi\mathrm{Im}(\widetilde{z}_2(t))}{\log |t|}\mod B\mathbb{Z}^{g_2}\in\Sigma_f,$$
where $\widetilde{z}=(\widetilde{z}_1,\widetilde{z}_2)\colon \Delta^*\to \mathbb{C}^{g_1}\times\mathbb{C}^{g_2}$ is a multi-valued holomorphic map lifting $z$.
This is well-defined, since replacing $\widetilde{z}$ by $\widetilde{z}+T_f(t)n$ for any $n=(n_1,n_2)\in\mathbb{Z}^{g_1}\times\mathbb{Z}^{g_2}$ yields
\begin{align*}
\mathrm{trop}(z)&\equiv-\lim_{t^m\to 0}\frac{2\pi \mathrm{Im}\left(\widetilde{z}_2(t^m)+\ltrans{S_3(t)}n_1+\frac{m\log t}{2\pi i} B n_2+S_2(t)n_2\right)}{\log |t^m|}\\
&\equiv-\lim_{t\to 0}\frac{2\pi\mathrm{Im}(\widetilde{z}_2(t))}{\log |t|}+Bn_2\equiv \mathrm{trop}(z)\mod B\mathbb{Z}^{g_2}.
\end{align*}
If the section $z\colon \Delta^*\to \mathscr{A}$ has a lift $\widetilde{z}\colon \Delta^*\to \mathbb{C}^g$ of the form $\widetilde{z}(t)=a+T_f(t)b$ for two real-valued vectors $a,b\in\mathbb{R}^g$ with $b=(b_1,b_2)\in\mathbb{R}^{g_1}\times\mathbb{R}^{g_2}$, we obtain by a similar computation
$$\mathrm{trop}(z)\equiv B b_2 \mod B\mathbb{Z}^{g_2}\in\Sigma_f.$$
In particular, by Equation (\ref{equ_valuetheta}) we may express the value of the modified Riemann theta function $\|\Psi_f\|$ associated to $\Sigma_f$ in $\mathrm{trop}(z)\in \Sigma_f$ as 
\begin{align}\label{equ_valuethetatrop}
\|\Psi_f\|(\mathrm{trop}(z))=\|\Psi_f\|(Bb_2)=\tfrac{1}{2}\min_{n\in\mathbb{Z}^{g_2}}\ltrans{(b_2+n)}B(b_2+n).
\end{align}

\section{Degeneration behavior of the theta function}\label{sec_theta}
In this section we would like to study the degeneration behavior of the normed Riemann $\theta$-function $\|\theta\|$ in families of abelian varieties. 
The Riemann theta function is given by
$$\theta\colon \mathbb{H}_g\times \mathbb{C}^g\to \mathbb{C},\quad \theta(\tau,z)=\sum_{n\in\mathbb{Z}^g}\exp(\pi i \ltrans{n}\tau n+2\pi i\ltrans{n}z).$$
We associate to it a real-valued version $\|\theta\|\colon \mathbb{H}_g\times \mathbb{C}^g\to \mathbb{R}_{\ge0}$ by
\begin{align*}
\|\theta\|(\tau,z)=\det(\mathrm{Im}~\tau)^{1/4}\exp(-\pi\ltrans(\mathrm{Im}~z)(\mathrm{Im}~\tau)^{-1}(\mathrm{Im}~z))\cdot |\theta|(z),
\end{align*}
which only depends on the class of $z$ in $\mathbb{C}^g/(\mathbb{Z}^g+\tau\mathbb{Z}^g)$ and which is invariant under the action of the symplectic group $\mathrm{Sp}_{2g}(\mathbb{Z})$. If $z=a+\tau b$ for some $a,b\in\mathbb{R}^g$, this can be written as
$$\|\theta\|(\tau,a+\tau b)=\det(\mathrm{Im}~\tau)^{1/4}\left|\sum_{n\in\mathbb{Z}^g}\exp(\pi i\ltrans{(n+b)}\tau(n+b)+2\pi i\ltrans{n}a)\right|.$$

Before we study the degeneration of $\|\theta\|$, we compute the limit of the determinant of the imaginary part of the period matrix. For two continuous functions $h_1, h_2\colon \Delta^*\to \mathbb{R}$ we write $h_1\sim h_2$ if $h_1-h_2$ can be extended to a continuous function on $\Delta$. 
\begin{Lem}\label{lem_period-limit}
	Let $T_f\colon \Delta^*\to\mathbb{H}_g$ be a multi-valued period map as in Equation (\ref{equ_period-matrix}). Then it holds
	$$\lim_{t\to 0}\frac{\det(\mathrm{Im}~T_f(t))}{(-\log|t|)^{\dim\Sigma_f}}=\left(\tfrac{1}{2\pi}\right)^{\dim\Sigma_f}\det(\mathrm{Im}~S_1(0))\det(B)>0.$$
	In particular, we have $\log\det(\mathrm{Im}~T_f(t))\sim \dim\Sigma_f\cdot \log(-\log|t|)$.
\end{Lem}
\begin{proof}
	Using Equation (\ref{equ_period-matrix}), we may expand the determinant of $\mathrm{Im}~T_f(t^m)$ by
	$$\det(\mathrm{Im}~T_f(t^m))=\left(\tfrac{1}{2\pi}\right)^{g_2}\det(\mathrm{Im}~S_1(t))\det(B)(-\log|t^m|)^{g_2}+\sum_{j=0}^{g_2-1}f_j(t)(\log|t|)^j,$$
	for some real-valued continuous functions $f_j$ on $\Delta$. As 
	$\mathrm{Im}~S_1(0)$ and $B$ are positive-definite, we obtain the first assertion of the lemma with $t$ replaced by $t^m$ after dividing by $(-\log|t^m|)^{g_2}$ on both sides. The second assertion follows by taking logarithms.
\end{proof}
Our main result in this section is the following proposition, which describes the degeneration behavior of $\|\theta\|$ and which will lead to a proof of Theorem \ref{mainthm} in the next section.
\begin{Pro}\label{proposition}
	Let $f\colon \mathscr{A}\to \Delta^*$ be a smooth family of principally polarized complex abelian varieties of dimension $g\ge 1$ and $T_f\colon \Delta^*\to \mathbb{H}_g$ a period map corresponding to this family. If $z\colon \Delta^*\to \mathscr{A}$ is a section of $f$ having a multi-valued lift $\widetilde{z}\colon \Delta^*\to \mathbb{C}^g$ given by $\widetilde{z}(t)=a+T_f(t)b$ for two real-valued vectors $a,b\in \mathbb{R}^g$, then there exists a non-zero real analytic function $\alpha\colon \mathbb{R}^{2g}\to \mathbb{R}_{\ge 0}$, such that
	 \begin{align}\label{equ_thetalimit}
	 \lim_{t\to 0}\left(\|\theta\|(T_f(t),z(t))\cdot|t|^{-\|\Psi\|(\mathrm{trop}(z))}(-\log |t|)^{-\dim \Sigma_f/4}\right)^2=\alpha(a,b).
	 \end{align}
\end{Pro}
\begin{proof}
	As the class of $a+T_f(t)b$ in $\mathscr{A}_t$ does not depend on the choice of the period map $T_f$, the value of
	$$\log\|\theta\|(T_f(t),z(t))$$
	does not depend on the choice of $T_f$, either. Hence, we may assume, that $T_f$ has the form as in Equation (\ref{equ_period-matrix}) and we use the notation as in Section \ref{sec_families}.
	Further, we shortly write $\mathbf{e}(x)=\exp(2\pi i x)$ for any $x\in\mathbb{C}$.
	We denote $a=(a_1,a_2)\in\mathbb{R}^{g_1}\times\mathbb{R}^{g_2}$ and $b=(b_1,b_2)\in\mathbb{R}^{g_1}\times\mathbb{R}^{g_2}$ for the splittings of the vectors $a$ and $b$ into their first $g_1$ and their last $g_2$ coordinates. 
	
	First, we would like to show, that the expression in the brackets on the left hand side of Equation (\ref{equ_thetalimit}) is bounded around $0\in\Delta$.
		For this purpose, we compute
		\begin{align}\label{computation}
		&\det(\mathrm{Im}~T_f(t^m))^{-1/4}\|\theta\|(T_f(t^m),z(t^m))=\left|\sum_{n\in\mathbb{Z}^{g}}\mathbf{e}\left(\tfrac{1}{2}\ltrans{(n+b)} T_f(t^m) (n+b) +\ltrans{n} a\right)\right|\\
		&=\left|\sum_{n_1\in\mathbb{Z}^{g_1},n_2\in\mathbb{Z}^{g_2}}t^{\frac{m}{2}(n_2+b_2)B(n_2+b_2)}\cdot\mathbf{e}(\ltrans{(n_1+b_1)}S_3(t)(n_2+b_2)+\ltrans{n_1}a_1+\ltrans{n_2}a_2)\right.\nonumber\\
		&\left. \qquad\quad\qquad\qquad\times\mathbf{e}\left(\tfrac{1}{2}\left(\ltrans{(n_1+b_1)} S_1(t) (n_1+b_1)+\ltrans{(n_2+b_2)}S_2(t)(n_2+b_2)\right)\right)\right|\nonumber\\
		&=\left|\sum_{n\in\mathbb{Z}^{g_2}}t^{\frac{m}{2}(n+b_2)B(n+b_2)}\theta(S_1(t),a_1+S_3(t)(n+b_2)+S_1(t)b_1)\right.\nonumber\\
		&\left.\qquad\qquad\times\mathbf{e}\left(\tfrac{1}{2}(\ltrans{b_1}S_1(t)b_1+\ltrans{(n+b_2)}S_2(t)(n+b_2))+\ltrans{b_1}S_3(t)(n+b_2)+\ltrans{n}a_2\right)\right|\nonumber\\
		&\le\sum_{n\in\mathbb{Z}^{g_2}}|t|^{\frac{m}{2}(n+b_2)B(n+b_2)}\det(\mathrm{Im}~S_1(t))^{-\frac{1}{4}}\|\theta\|(S_1(t),a_1+S_3(t)(n+b_2)+S_1(t)b_1)\nonumber\\
		&\qquad\qquad\times\exp(\pi\ltrans{(n+b_2)}(\ltrans{(\mathrm{Im}~S_3(t))}(\mathrm{Im}~S_1(t))^{-1}(\mathrm{Im}~S_3(t))-\mathrm{Im}~S_2(t))(n+b_2)),\nonumber
		\end{align}
		where the last inequality just follows from the triangle inequality and the definition of $\|\theta\|$.
		As $\det(\mathrm{Im}~S_1(t))^{-1/4}\|\theta\|(S_1(t),a_1+S_3(t)(n+b_2)+S_1(t)b_1)$ is continuous and well-defined at $t=0$, there exists an open neighborhood $0\in U_1\subseteq \Delta$ and a constant $M_1\in\mathbb{R}$ such that
		$$\det(\mathrm{Im}~S_1(t))^{-1/4}\|\theta\|(S_1(t),a_1+S_3(t)(n+b_2)+S_1(t)b_1)\le M_1$$
		for all $t\in U_1$. Note that $\|\theta\|(S_1(t),\cdot)$ can be seen as a continuous function on the compact space $\mathbb{C}^{g_1}/(\mathbb{Z}^{g_1}+S_1(t)\mathbb{Z}^{g_1})$, where it has to attain its supremum. Hence, we can choose $M_1$ independently of the choice of $n\in\mathbb{Z}^{g_2}$. 
		
		We shortly write $T=\frac{1}{2}\min_{n\in\mathbb{Z}^{g_2}}(n+b_2)B(n+b_2)$. Note, that $T=\|\Psi_f\|(\mathrm{trop}(z))$ by Equation (\ref{equ_valuethetatrop}). We define the bounded, discrete and hence finite sets $$N_k=\left\{n\in\mathbb{Z}^{g_2}~|~\tfrac{1}{2}(n+b_2)B(n+b_2)\le kT\right\}.$$
		Further, we shortly write $S(t)=\ltrans{(\mathrm{Im}~S_3(t))}(\mathrm{Im}~S_1(t))^{-1}(\mathrm{Im}~S_3(t))-\mathrm{Im}~S_2(t)$.
		If we assume $t\in U_1\setminus\{0\}$, the above calculation can be continued by
		\begin{align}\label{continuation}
		&\det(\mathrm{Im}~T_f(t^m))^{-1/4}\|\theta\|(T_f(t^m),z(t^m))\cdot |t|^{-mT}\\
		&\le M_1\left(\sum_{n\in N_2}|t|^{m\left(\frac{1}{2}(n+b_2)B(n+b_2)-T\right)}\exp(\pi\ltrans{(n+b_2)}S(t)(n+b_2))\right.\nonumber\\
		&\left.\qquad +\sum_{n\in \mathbb{Z}^{g_2}\setminus N_2}|t|^{m\left(\frac{1}{2}(n+b_2)B(n+b_2)-T\right)}\exp(\pi\ltrans{(n+b_2)}S(t)(n+b_2))\right)\nonumber
		\end{align}
		Since the first sum is a finite sum and its addends are continuous and well-defined in $t=0$, there exists an open neighborhood $0\in U_2\subseteq U_1$ and a constant $M_2$ such that
		$$  \sum_{n\in N_2}|t|^{m\left(\frac{1}{2}(n+b_2)B(n+b_2)-T\right)}\exp(\pi\ltrans{(n+b_2)}S(t)(n+b_2))\le M_2$$
		for all $t\in U_2$.
		
		To estimate the second sum, let us first remark, that we have 
		$$\tfrac{1}{2}(n+b_2)B(n+b_2)-T> \tfrac{1}{4}(n+b_2)B(n+b_2)$$
		for $n\in\mathbb{Z}^{g_2}\setminus N_2$.
		As $S(t)$ and $B$ are symmetric, their eigenvalues are real. We denote $\lambda_{\mathrm{max}}(S(t))$ for the greatest eigenvalue of $S(t)$ and $\lambda_{\mathrm{min}}(B)$ for the smallest eigenvalue of $B$. We fix some $\epsilon>0$. As $B$ is positive definite, it can be checked by a simultaneous diagonalization, that $\lambda\cdot B-S(t)\ge \epsilon\cdot \mathrm{Id}_n$ for all $\lambda\ge \frac{\lambda_{\mathrm{max}}(S(t))+\epsilon}{\lambda_{\mathrm{min}}(B)}$. By continuity there exists an open neighborhood $0\in U_3\subseteq U_2$ with $\lambda\cdot B-S(t)\ge \epsilon\cdot \mathrm{Id}_n$ for all $t\in U_3$ and all $\lambda\ge \frac{\lambda_{\mathrm{max}}(S(0))+2\epsilon}{\lambda_{\mathrm{min}}(B)}$. By shrinking $U_3$, we may assume, that 
		$$-\log|t|\ge \frac{4\pi}{m}\left(\frac{\lambda_{\mathrm{max}}(S(0))+2\epsilon}{\lambda_{\mathrm{min}}(B)}\right)$$
		for all $t\in U_3\setminus \{0\}$.
		Now we obtain for the second sum in (\ref{continuation})
		\begin{align*}
		&\sum_{n\in \mathbb{Z}^{g_2}\setminus N_2}|t|^{m(\frac{1}{2}(n+b_2)B(n+b_2)-T)}\exp(\pi\ltrans{(n+b_2)}S(t)(n+b_2))\\
		&<\sum_{n\in \mathbb{Z}^{g_2}\setminus N_2}\exp\left(\log|t|\tfrac{m}{4}\ltrans{(n+b_2)}B(n+b_2)\right)\exp(\pi\ltrans{(n+b_2)}S(t)(n+b_2))\\
		&=\sum_{n\in \mathbb{Z}^{g_2}\setminus N_2}\exp\left(-\pi\ltrans{(n+b_2)}\left(-\tfrac{m\log|t|}{4\pi}B-S(t)\right)(n+b_2)\right)\\
		&\le\sum_{n\in \mathbb{Z}^{g_2}}\prod_{i=1}^{g_2} \exp\left(-\pi \epsilon(n_i+b_{2,i})^2\right)\le \left(2\sum_{n=0}^\infty\exp(-\pi\epsilon n)\right)^{g_2}=\left(\frac{2}{1-e^{-\pi\epsilon}}\right)^{g_2}
		\end{align*}
		for all $t\in U_3\setminus\{0\}$. Setting $M=M_1\left(M_2+\left(\frac{2}{1-e^{-\pi\epsilon}}\right)^{g_2}\right)$, we conclude
		$$\det(\mathrm{Im}~T_f(t^m))^{-1/4}\|\theta\|(T_f(t^m),z(t^m))\cdot |t|^{-mT}
		\le M$$
		for all $t\in U_3\setminus\{0\}$. 
		Hence, we deduce from the computation in (\ref{computation}), that the left hand side tends to $|\beta(a,b)|$, where
		 \begin{align}\label{alpha}
		 \beta(a,b)=&\sum_{n\in N_1}\mathbf{e}\left(\tfrac{1}{2}(\ltrans{b_1}S_1(0)b_1+\ltrans{(n+b_2)}S_2(0)(n+b_2))+\ltrans{b_1}S_3(0)(n+b_2)+\ltrans{n}a_2\right)\\
		 &\qquad\times\theta(S_1(0),a_1+S_3(0)(n+b_2)+S_1(0)b_1).\nonumber
		 \end{align}
		 
		 Note, that the real part and the imaginary part of $\beta(a,b)$ are real analytic functions in $(a,b)\in\mathbb{R}^{2g}$ and that $\beta$ is not identically zero, as $\theta(S_1(0),\cdot)$ is a non-zero holomorphic function. Hence, the function
		 $$\alpha(a,b)=\left(\frac{\det(\mathrm{Im}~S_1(0)B)}{(2\pi)^{\dim\Sigma_f}}\right)^{1/2}(\mathrm{Re}~\beta(a,b))^2+(\mathrm{Im}~\beta(a,b))^2)$$
		 is a non-negative and non-zero real analytic function on $\mathbb{R}^{2g}$ and it holds
		 \begin{align*}
		 &\lim_{t\to 0}\left(\|\theta\|(T_f(t),z(t))|t|^{-\|\Psi\|(\mathrm{trop}(z))}(-\log |t|)^{-\dim \Sigma_f/4}\right)^2\\
		 =&\left(\frac{\det(\mathrm{Im}~S_1(0)B)}{(2\pi)^{\dim\Sigma_f}}\right)^{1/2}\lim_{t^m\to 0}\left(\|\theta\|(T_f(t^m),z(t^m))|t|^{-mT}\det(\mathrm{Im}~T_f(t^m))^{-1/4} \right)^2\\
		 =&\left(\frac{\det(\mathrm{Im}~S_1(0)B)}{(2\pi)^{\dim\Sigma_f}}\right)^{1/2}|\beta(a,b)|^2=\alpha(a,b),
		 \end{align*}
		 where the first equality follows by Lemma \ref{lem_period-limit}. This proves the proposition.
\end{proof}
\section{Proof of Theorem \ref{mainthm}}\label{sec_proof}
We give the proof of Theorem \ref{mainthm} in this section. We first clarify some notations. Let $(A,\Theta)$ be a principally polarized complex abelian variety, where $\Theta\subseteq A$ denotes a divisor, such that $\mathcal{O}_A(\Theta)$ is an ample and symmetric line bundle satisfying $\dim H^0(A,\mathcal{O}_A(\Theta))=1$ and inducing the principal polarization. Then there exists a matrix $\tau\in\mathbb{H}_g$ such that $A\cong \mathbb{C}^g/(\mathbb{Z}^g+\tau\mathbb{Z}^g)$ and $\Theta$ is the zero-divisor of $\theta(\tau,\cdot)$. The canonical 1-1 form of $(A,\Theta)$ is given by
\begin{align}\label{equ_form}
\nu=\tfrac{i}{2}\sum_{j,k=1}^g ((\mathrm{Im}~\tau)^{-1})_{j k} dz_j\wedge d\overline{z}_k,
\end{align}
where $z_1,\dots,z_g$ are the coordinates in $\mathbb{C}^g$. Its $g$-th power is equal to
$$\nu^g=\frac{i^g g!}{2^g\det(\mathrm{Im}~\tau)}dz_1\wedge d\overline{z}_1\wedge\dots\wedge dz_g\wedge d\overline{z}_g.$$
We would like to substitute real coordinates of the form $z=a+\tau b$ with $a,b\in\mathbb{R}^g$.
By a direct computation we obtain
$$dz_1\wedge d\overline{z}_1\wedge\dots\wedge dz_g\wedge d\overline{z}_g=(2i)^g \det (\mathrm{Im}~\tau) da_1\wedge db_1\wedge\dots\wedge da_g\wedge db_g$$
and hence, $\nu^g=(-1)^g g!da_1\wedge db_1\wedge\dots\wedge da_g\wedge db_g$. Thus, we obtain $\nu^g(a+\tau b)=g!\lambda(a,b)$ for the Lebesgue measure $\lambda$ on $\mathbb{R}^{2g}$, as both volume forms are positive.

 To $(A,\Theta)$ we associate the invariant
\begin{align*}
I(A,\Theta)&=\log\int_A\|\theta\|^2(\tau,z)\frac{\nu^g(z)}{g!}-\int_A\log\|\theta\|^2(\tau,z)\frac{\nu^g(z)}{g!}\\
&=-\frac{g\log 2}{2}-\int_A\log\|\theta\|^2(\tau,z)\frac{\nu^g(z)}{g!}.
\end{align*}
We refer to \cite[Proposition 8.5.6]{BL04} for the computation of the first integral. Note, that we have $I(A,\Theta)=-2H(A,\Theta)-\frac{g\log2}{2}$ for the invariant $H(A,\Theta)$ defined in \cite[Section 2.1]{Wil17}.
As $[0,1)\times \tau[0,1)$ is a fundamental domain in $\mathbb{C}^g$ for $A$, we can express $I(A,\Theta)$ by
$$I(A,\Theta)=-2\int_{(a,b)\in[0,1)^{2g}}\log\|\theta\|(\tau,a+\tau b)\lambda(a,b)-\frac{g\log 2}{2}.$$
Now we give the proof of Theorem \ref{mainthm}.
\begin{proof}[Proof of Theorem \ref{mainthm}]
	Again, we use the notation from Section \ref{sec_families} associated to a smooth family $f\colon \mathscr{A}\to\Delta^*$ of principally polarized complex abelian varieties of dimension $g\ge 1$. 
	Let $\alpha$ be the non-zero real analytic function in Proposition \ref{proposition} and $W$ its zero set. By \cite{Mit15} $W$ has Lebesgue measure $0$.
	We compute by Proposition \ref{proposition} and Equations (\ref{equ_tropmoment}) and  (\ref{equ_valuethetatrop})
	\begin{align*}
	&\lim_{t\to 0}\left(I(\mathscr{A}_{t})+\tfrac{g\log 2}{2}+I(\Sigma_f)\log|t|+\frac{\dim \Sigma_f}{2}\log(-\log|t|)\right)\\
	&=\int_{(a,b)\in [0,1)^{2g}\setminus W}\lim_{t\to 0}\left(-\log\|\theta\|(T_f(t),a+T_f(t)b)^2+2\|\Psi\|(Bb_2)\log|t|\right.\\
	&\qquad\qquad\qquad\qquad\qquad\left.+\frac{\dim \Sigma_f}{2}\log(-\log|t|)\right)\lambda(a,b)\\
	&=-\int_{(a,b)\in [0,1]^{2g}}\log \alpha(a,b)\lambda(a,b).\\
	\end{align*}
	By Lemma \ref{Lem_analytic} below the function $\log\alpha$ is locally integrable. Thus, we can write
	$$\lim_{t\to 0}\left(I(\mathscr{A}_{t})+I(\Sigma_f)\log|t|+\frac{\dim\Sigma_f}{2}\log(-\log|t|)\right)=M$$
	for some $M\in \mathbb{R}$. This proves the theorem.
\end{proof}

\begin{Lem}\label{Lem_analytic}
	For any non-zero analytic function $f\colon \mathbb{R}^n\to \mathbb{R}$ the function $\log|f|$ is locally integrable.
\end{Lem}
\begin{proof}
	Let $p\in \mathbb{R}^n$. If $f(p)\neq 0$, then there exist an open neighborhood $p\in U_p$ of finite volume $V_p$ and a real number $M_p\in \mathbb{R}$, such that $\log|f(x)|>M_p$ for all $x\in U_p$. Therefore, we have $\int_{U_p}\log|f(x)|dx\ge M_p \cdot V_p$.
	
	Thus, we may assume $f(p)=0$.  By Hironaka's resolution theorem there exist an open neighborhood $p\in U_p\in\mathbb{R}^n$, a real analytic manifold $\widetilde{U}_p$ and a proper analytic map $\varphi\colon \widetilde{U}_p\to U_p$ such that 
	$$\varphi\colon \widetilde{U}_p\setminus\{f\circ \varphi=0\}\to U_p\setminus\{f=0\}$$
	is an isomorphism and for each $q\in \widetilde{U}_p$ there exist local analytic coordinates $(y_1,\dots,y_n)$ centered at $q$, such that $f\circ \varphi=\epsilon_q \prod_{j=1}^n y_j^{k_j(q)}$ for some $k_j(q)\in \mathbb{Z}_{\ge 0}$ and an invertible analytic function $\epsilon_q$ in an open neighborhood $q\in \widetilde{U}_{p,q}\subseteq \widetilde{U}_p$. We refer to \cite{Ati70} for this version of Hironaka's resolution theorem. By shrinking $\widetilde{U}_{p,q}$ and scaling the coordinates, we may assume that $\widetilde{U}_{p,q}=\{x\in \widetilde{U}_p~|~ -1<y_i(x)<1\}$ lies in a compact subset of $\widetilde{U}_p$ and that there is a negative constant $C_q\in \mathbb{R}_{<0}$ such that $\log |\epsilon_q(x)|\ge C_q$ for all $x\in \widetilde{U}_{p,q}$. Further, we may restrict $U_p$ such that $|f(x)|<1$ for all $x\in U_p$.
	
	We can choose a compact neighborhood $K$ of $p$ such that $K\subseteq U_p$. As $\varphi$ is proper, the pre-image $\varphi^{-1}(K)\subseteq \widetilde{U}_p$ is also compact. Hence, there is a finite set of points $\{q_1,\dots,q_l\}\subseteq \widetilde{U}_p$, such that
	$$\varphi^{-1}(K)\subseteq\bigcup_{i=1}^l \widetilde{U}_{p,q_i}.$$
	Let $\left(\rho_i\colon \widetilde{U}_{p,q_i}\to [0,1]\right)_{1\le i\le l}$ be an associated partition of unity. Then we may compute the integral of $\log|f|$ over $K$ by a pullback along $\varphi$
	\begin{align*}
	&\int_K \log|f(x)| dx_1\dots dx_n\\
	&=\sum_{i=1}^l \int_{\widetilde{K}_i}\left(\log|\epsilon_{q_i}(y)|+\sum_{j=1}^n k_j(q_i)\log|y_j|\right) \rho_i(y) \det \left(\frac{\partial \varphi_k(y)}{\partial y_j}\right)dy_1\dots dy_n,
	\end{align*}
	where $\widetilde{K}_i=\widetilde{U}_{p,q_i}\cap \varphi^{-1}(K)$.
	As $\varphi$ is analytic and $\widetilde{U}_{p,q_i}$ lies in a compact subset of $\widetilde{U}_p$, there is a real number $C_1$, such that
	$$\left|\det \left(\frac{\partial \varphi_j(y)}{\partial y_k}\right)\right|\le C_1.$$
	We set $C=2^n\sum_{i=1}^{l}C_{q_i} C_1$ and $k=\sum_{i=1}^l\sum_{j=1}^n k_j(q_i)$. As we always have $|y_j|<1$ on $\widetilde{U}_{p,q_i}$, we can bound
	$$\int_K \log|f(x)| dx_1\dots dx_n\ge C+2^{n-1}kC_1\int_{-1}^{1}\log|t|dt=C-2^n kC_1>-\infty.$$
	Since $K$ was chosen to be a compact neighborhood of $p$, this means that $\log|f|$ is locally integrable.
\end{proof}

\section{Metrized graphs and tropical Jacobians}\label{sec_graphs}
Before we apply our main theorem to obtain arithmetic applications, we discuss the notions of metrized graphs and their associated tropical Jacobians in this section. In particular, we discuss potential theory on metrized graphs as introduced by Zhang \cite{Zha93} and Chinburg--Rumely \cite{CR93}, see also \cite{BR07}.

By a \emph{metrized graph} we mean a compact and connected metric space $\Gamma$, such that for all $p\in\Gamma$ there exists an open neighborhood $p\in U_p\subseteq \Gamma$, which is isometric to the star-shaped set
$$S(n_p,r_p)=\{z\in \mathbb{C}~|~z=t e^{2\pi i k/n_p}\text{ for some } 0\le t\le r_p\text{ and } k\in\mathbb{Z}\},$$
for some $n_p\in\mathbb{Z}_{> 0}$ and $r_p\in \mathbb{R}_{> 0}$ or to a point, in which case we set $n_p=0$. We call $n_p$ the valency of $p$. There are only finitely many points $p\in\Gamma$ of valency $n_p\neq 2$ and we denote this set of points by $V_0$. In general, a finite set $V\subseteq \Gamma$ is called a \emph{vertex set} if $V_0\subseteq V$ and the closure of any connected component of $\Gamma\setminus V$ intersects $V$ in exactly two points. For the following we fix a vertex set $V$.
The space $\Gamma\setminus V$ consists of finitely many open line segments $e_1,\dots, e_r\subseteq \Gamma$. We write $\ell(e)$ for the length of any line segment $e$ and we define 
$$\delta(\Gamma)=\sum_{i=1}^r \ell(e_i).$$

Next, we would like to consider potential theory on a metrized graph $\Gamma$. A path in $\Gamma$ is a length-preserving continuous map $\gamma\colon [0,l]\to \Gamma$ for some $l>0$. The tangent space of $\Gamma$ at a point $p\in \Gamma$ consists of the equivalence classes
$$T_p(\Gamma)=\{\gamma \text{ path in } \Gamma \text{ with }\gamma(0)=p\}/\sim,$$
where $\gamma\sim\gamma'$ holds if there exists some $\epsilon>0$ such that $\gamma(t)=\gamma'(t)$ for all $0\le t<\epsilon$. Note that $\#T_p(\Gamma)=n_p$. For any tangent vector $v\in T_p(\Gamma)$ represented by a path $\gamma$ we define the one-sided directional derivative at $p$ as
$$d_v f(p)=\lim_{t\to 0^+}\frac{f(\gamma(t))-f(p)}{t}$$
if the limit exists. For the following class of functions
\begin{align*}
\mathrm{Zh}(\Gamma)=\{f\colon \Gamma\to\mathbb{R}~|~&f \text{ continuous, piecewise } C^2 \text{ and} \\
&d_v f(p) \text{ exists for all } p\in\Gamma, v\in T_p(\Gamma)\}
\end{align*}
Zhang \cite[Appendix]{Zha93} introduced the Laplacian operator
$$\Delta(f)=f''dx+\sum_{p\in\Gamma}\left(\sum_{v\in T_p(\Gamma)}d_v f(p)\right)\delta_p,$$
which is considered as a measure on $\Gamma$, where $\delta_p$ denotes the Dirac measure at $p$.

Let $\mu$ be a measure on $\Gamma$ of volume $1$. For every point $x\in \Gamma$ we obtain a unique function $g_\mu(x,\cdot)\in \mathrm{Zh}(\Gamma)$ satisfying
$$\Delta_y g_\mu(x,y)=\delta_y-\mu \quad\text{and}\quad \int_{\Gamma}g_\mu(x,y)\mu(y)=0.$$
This gives a unique continuous and symmetric function $g_\mu\colon \Gamma\times\Gamma\to \mathbb{R}$ called the \emph{Green function} associated to $\mu$. For $\mu=\delta_p$ we particularly write $g_p(x,y)=g_{\delta_p}(x,y)$ and we call $r(\Gamma;p,q)=g_p(q,q)$ the \emph{resistance function}, which is also symmetric in $p$ and $q$. 
Let $e_i$ be any open line segment as above. We write $p_i,p_i'$ for the points in the intersection $\overline{e}_i\cap V$. Then we define
$$r(e_i)=\begin{cases} r(\Gamma\setminus e_i; p_i,p_i') & \text{if } \Gamma\setminus e_i \text{ is connected,}\\
\infty & \text{else.}\end{cases}$$

A divisor on $\Gamma$ is a finite formal sum $D=\sum_{p\in\Gamma}m_p p$. The canonical divisor on $\Gamma$ is given by $K_{\mathrm{can}}=\sum_{p\in\Gamma}(n_p-2)p$. Let $K=\sum_{p\in\Gamma}(n_p+m_p-2)p$ be an effective divisor on $\Gamma$ for some integers $m_p\ge 0$. The pair $(\Gamma,K)$ is also called a \emph{polarized metrized graph} of genus $g=\frac{1}{2}\deg K+1$. To distinguish it from the genus of the underlying graph, we write $g_0(\Gamma)=g_0(\Gamma,K)$ for the genus of the graph $\Gamma$ forgetting the polarization. If $g\neq 0$, the canonical Zhang measure $\mu$ of $(\Gamma,K)$ is given by
$$\mu=\frac{1}{2g}\left(\delta_K-\delta_{K_{\mathrm{can}}}+\sum_{i=1}^r\frac{2}{r(e_i)+\ell(e_i)}dx|_{e_i}\right).$$

Now we can give the definitions of the invariants associated to the polarized metrized graph $(\Gamma, K)$, which we will need in the subsequent sections. The $\epsilon$-invariant introduced by Zhang in \cite[Theorem 4.4]{Zha93} is defined by
$$\epsilon(\Gamma,K)=\int_{\Gamma} g_{\mu}(x,x)((2g-2)\mu+\delta_{K}).$$
This invariant is always non-negative.
The $\varphi$-invariant was introduced by Zhang in \cite[Theorem 1.3.1]{Zha10} and it has the form
$$\varphi(\Gamma, K)=-\tfrac{1}{4}\delta(\Gamma)+\tfrac{1}{4}\int_\Gamma g_\mu(x,x)((10g+2)\mu-\delta_K).$$

There are several relations between these invariants. First, we would like to recall the inequality 
$$\delta(\Gamma)\le \tfrac{2g(7g+5)}{(g-1)^2}\varphi (\Gamma, K)$$
for $g\ge 2$ due to Cinkir \cite[Theorem 2.11]{Cin11}. To give a second inequality, let us recall the invariant $\tau(\Gamma)$ introduced by Baker--Rumely \cite[Theorem 14.1]{BR07} based on a result by Chinburg--Rumely \cite[Theorem 2.11]{CR93}, which is non-negative by \cite[Lemma 14.4]{BR07}. Due to de Jong \cite[Proposition 9.2]{dJo18} it can be expressed as 
$$0\le \tau(\Gamma)=\tfrac{1}{12}\left(\delta(\Gamma)+4\varphi(\Gamma,K)-2\epsilon(\Gamma,K)\right).$$
If we combine both inequalities, we obtain for $g\ge 2$
\begin{align}\label{equ_bound-of-varphi}
0\le\delta(\Gamma)+\epsilon(\Gamma,K)\le\tfrac{3}{2}\delta(\Gamma)+2\varphi(\Gamma,K)\le\tfrac{23g^2+11g+2}{(g-1)^2}\varphi(\Gamma, K).
\end{align}

Next we would like to define the tropical Jacobian associated to $\Gamma$.
Forgetting the metric we may consider $\Gamma$ as a usual directed graph $\Gamma$ with vertex set $V(\Gamma)=V$ and edge set $E(\Gamma)=\{e_1,\dots,e_r\}$, where we fix an orientation on the edges, such that $p_i$ is the source and $p_i'$ the target of $e_i$. The group of $0$-dimensional chains is given by $C_0(\Gamma,\mathbb{Z})=\bigoplus_{p\in V}\mathbb{Z}p$ and the group of $1$-dimensional chains by $C_1(\Gamma,\mathbb{Z})=\bigoplus_{i=1}^r\mathbb{Z}e_i$. The boundary map is defined by
$$\delta\colon C_1(\Gamma)\to C_0(\Gamma),\quad \sum_{i=1}^r m_i e_i\mapsto \sum_{i=1}^r m_i (p'_i-p_i)$$
and we set $H_1(\Gamma,\mathbb{Z})=\ker \delta$. The metric of $\Gamma$ induces an inner product $[\cdot,\cdot]_\Gamma$ on $C_1(\Gamma,\mathbb{R})=C_1(\Gamma,\mathbb{Z})\otimes_{\mathbb{Z}}\mathbb{R}$ by $[e_i,e_j]_\Gamma=\delta_{ij} \ell(e_i)$, where $\delta_{ij}$ denotes the Kronecker symbol. This inner product restricts to an inner product $[\cdot,\cdot]_\Gamma$ on $H_1(\Gamma,\mathbb{R})=H_1(\Gamma,\mathbb{Z})\otimes_\mathbb{Z}\mathbb{R}$. Hence, $H_1(\Gamma,\mathbb{Z})$ together with $[\cdot,\cdot]_\Gamma$ forms an Euclidean lattice and we define the \emph{tropical Jacobian} of $\Gamma$ to be the associated polarized real torus $\mathrm{Jac}(\Gamma)=H_1(\Gamma,\mathbb{R})/H_1(\Gamma,\mathbb{Z})$. Note, that $\mathrm{Jac}(\Gamma)$ does not depend on the choice of the vertex set $V$ and the orientation of the edges.
Due to de Jong and Shokrieh \cite[Equation (1.4)]{dJS18b} the tropical moment of $\mathrm{Jac}(\Gamma)$ can be related to the above invariants of $\Gamma$ by the formula
\begin{align}\label{equ_delta+epsilon}
\delta(\Gamma)+\epsilon(\Gamma, K)=12 I(\mathrm{Jac}(\Gamma))+2\varphi(\Gamma, K).
\end{align}
By construction we have $g_0(\Gamma)=\dim \mathrm{Jac}(\Gamma)$.
\section{Families of Riemann surfaces}\label{sec_families_riemann}
We would like to connect Theorem \ref{mainthm} to de Jong's result of the degeneration behavior of Faltings' $\delta$-invariant on families of Riemann surfaces in \cite{dJo19}. In particular, we recall the metrized graph associated to such a family and we compare the polarized real torus associated to the corresponding family of their Jacobian varieties in Section \ref{sec_families} with the tropical Jacobian associated to this metrized graph. As an application we will prove Corollaries \ref{cor_ZK} and \ref{cor_lowerbound}.

First, we give definitions for the Zhang--Kawazumi invariant $\varphi$ and for Faltings' $\delta$-invariant. Let $M$ be any compact and connected Riemann surface of genus $g\ge 2$. Its Jacobian $\mathrm{Jac}(M)$ is a principally polarized complex abelian variety. Let $\nu$ be its canonical 1-1 form defined in (\ref{equ_form}). We define the 1-1 form $\mu=\frac{1}{g}I^*\nu$ on $M$, where $I\colon M\to \mathrm{Jac}(M)$ denotes the Abel--Jacobi embedding associated to some base point. The Arakelov--Green function $G\colon M^2\to \mathbb{R}_{\ge 0}$ is characterized by
$$\partial_x \overline{\partial}_x \log G(x,y)=\pi i (\mu(x)-\delta_y(x)),\quad \int_M \log G(x,y)\mu(x)=0.$$
Let $\Delta\subseteq M^2$ denote the diagonal and $h_\Delta$ the curvature form of the line bundle $\mathcal{O}_{M^2}(\Delta)$. The Zhang--Kawazumi invariant
$$\varphi(M)=\int_{M^2} \log G(x,y) h_\Delta^2(x,y)$$
has been introduced independently by Zhang \cite[Theorem 1.3.1]{Zha10} and Kawazumi \cite{Kaw08}. We refer to the proof of \cite[Proposition 2.5.3]{Zha10} for this expression. From the same proposition one deduces that one always has $\varphi(M)>0$ as $g\ge 2$.
Faltings' $\delta$-invariant has been introduced by Faltings in \cite[p. 402]{Fal84} and as shown in \cite[Theorem 1.1]{Wil17} it can be expressed by
\begin{align}\label{equ_delta}
\delta(M)=12 I(\mathrm{Jac}(M))+2\varphi(M)-2g\log 2\pi^4.
\end{align}
Note, that Equation (\ref{equ_delta+epsilon}) could be seen as a tropical version of this equation.

Now let $\pi\colon \mathcal{X}\to \Delta$ be a semistable curve of genus $g\ge 2$ over the complex unit disc $\Delta$ and set $\mathcal{X}^*=\pi^{-1}(\Delta^*)$ as well as $\mathcal{X}_0=\pi^{-1}(0)$. We assume that $\mathcal{X}^*$ is smooth over $\Delta^*=\Delta\setminus\{0\}$.
We associate to $\mathcal{X}_0$ its dual graph $\Gamma_\pi$, whose vertex set $V(\Gamma_\pi)$ consists of the irreducible components of $\mathcal{X}_0$ and whose edge set $E(\Gamma_\pi)$ consists of the singular points in $\mathcal{X}_0$, such that any edge connects the two not necessarily different vertices associated to the irreducible components meeting in the corresponding singular point.

We obtain a metrization of $\Gamma_\pi$ by the family $\pi$. To describe this, it is enough to assign a length $\ell(e)$ to every edge $e$ of $\Gamma_\pi$. Let $p_e\in\mathcal{X}_0$ be the singular point corresponding to an edge $e$ of $\Gamma_\pi$. Then there is an open neighborhood $p_e\subseteq \mathcal{U}_e\subseteq\mathcal{X}$ such that $\mathcal{U}_e$ is isomorphic to an open neighborhood of $0$ in the subspace $\mathcal{Z}_n=\{(x,y,t)\in\mathbb{C}^3~|~xy=t^n\}$ of $\mathbb{C}^3$ for some $n\in\mathbb{Z}_{\ge 1}$, see for example \cite[Proposition X.2.1]{ACG11}. We set $\ell(e)=n$. We also obtain a polarization on $\Gamma_\pi$ by the effective divisor
$$K=\sum_{v\in V(\Gamma_\pi)} (n_v+g_v-2)v,$$
where $n_v$ denotes the valency of the vertex $v$ as in Section \ref{sec_graphs} and $g_v$ denotes the genus of the normalization of the irreducible component of $\mathcal{X}_0$ corresponding to $v\in V(\Gamma_\pi)$.
The polarized metrized graph $\Gamma_\pi=(\Gamma_\pi,K)$ has genus $g$.

Now the main result of \cite{dJo19} states, that if $\Omega(t)$ is a family of normalized period matrices on $\Delta^*$ determined by a symplectic framing of $R^1\pi_*\mathbb{Z}_{\pi^{-1}(\Delta^*)}$, then we have
\begin{align*}
\delta(\mathcal{X}_t)\sim -(\delta(\Gamma_\pi)+\epsilon(\Gamma_\pi))\log|t| -6\log\det\mathrm{Im}~\Omega(t).
\end{align*}
According to \cite[Proposition 11.2]{dJo19} one can choose the family $\Omega(t)$ to have the form as in (\ref{equ_period-matrix}). Hence, we may apply Lemma \ref{lem_period-limit} to obtain
\begin{align}\label{equ_delta-dejong}
\delta(\mathcal{X}_t)\sim -(\delta(\Gamma_\pi)+\epsilon(\Gamma_\pi))\log|t| -6\dim\Sigma_f\log(-\log|t|),
\end{align}
where $f\colon \mathrm{Jac}(\mathcal{X}^*)\to \Delta^*$ denotes the family of principally polarized complex abelian varieties obtained by taking the Jacobian of $\mathcal{X}^*$.
We would like to compare de Jong's formula with Theorem \ref{mainthm}. To apply the results in Section \ref{sec_graphs}, especially Equation (\ref{equ_delta+epsilon}), we need to show, that the associated polarized real torus $\Sigma_f$ coincides with tropical Jacobian $\mathrm{Jac}(\Gamma_\pi)$.
\begin{Lem}\label{lem_period}
It holds $\Sigma_f\cong \mathrm{Jac}(\Gamma_{\pi})$. In particular, we have $g_0(\Gamma_\pi)=\dim \Sigma_f$.
\end{Lem}
\begin{proof}
We fix an orientation on the edges of $\Gamma_\pi$.
Let $r=g_0(\Gamma_\pi)$ denote the genus of the graph $\Gamma_\pi$. We may choose edges $e_1,\dots, e_r$ of $\Gamma_\pi$, such that $\Gamma_\pi\setminus \{e_1,\dots,e_r\}$ is a connected tree. Further, we may choose a basis $B_1,\dots,B_r$ of $H_1(\Gamma,\mathbb{Z})$ with $B_i=\sum_{e\in E(\Gamma_\pi)}b_{i,e}e$ for some $b_{i,e}\in\{-1,0,1\}$ and $b_{i,e_j}=\delta_{ij}$, where $\delta_{ij}$ denotes the Kronecker symbol. We have a short exact sequence
$$0\to H_1(\widetilde{\mathcal{X}}_0,\mathbb{Z})\to H_1(\mathcal{X}_0,\mathbb{Z})\to H_1(\Gamma_\pi,\mathbb{Z})\to 0,$$
where $\widetilde{\mathcal{X}}_0$ denotes the normalization of $\mathcal{X}_0$. Moreover, any embedding of $\Gamma_\pi$ into $\mathcal{X}_0$ defines an isomorphism
$$H_1(\mathcal{X}_0,\mathbb{Z})\cong H_1(\widetilde{\mathcal{X}}_0,\mathbb{Z})\oplus H_1(\Gamma_\pi,\mathbb{Z}).$$
Hence, we may construct a basis of $H_1(\mathcal{X}_0,\mathbb{Z})$ by the union of the basis $B_1,\dots,B_r$ and a symplectic basis $A_{r+1},\dots,A_g,B_{r+1},\dots, B_g$ of $H_1(\widetilde{\mathcal{X}}_0,\mathbb{Z})$.

We write $p_e$ for the singular point in $\mathcal{X}_0$ corresponding to an edge $e\in E(\Gamma_\pi)$. Recall that $\mathcal{X}$ was locally around $p_e$ given by the equation $xy=t^{\ell(e)}$ as a subspace of $\mathbb{C}^3$. Hence, there is an open neighborhood $0\in U\subseteq\Delta$, such that 
$$A_{e,t}=\{(x,y,z)\in  \mathbb{C}^2\times \{t\}~|~|x|=|t|^{\ell(e)/2},xy=t^{\ell(e)}\}$$
is well-defined and contained in $\mathcal{X}_t$ for any $t\in U$ and any $e\in E(\Gamma_\pi)$. In particular, we denote $A_{i,t}=A_{e_i,t}$ for $1\le i\le r$. For any domain $D\subseteq U\setminus \{0\}$ we may deform the cycles $A_{r+1},\dots,A_g,B_{1},\dots, B_g$ of $H_1(\mathcal{X}_0,\mathbb{Z})$ and choose an orientation on $A_{i,t}$ to obtain a symplectic basis 
$$A_{1,t},\dots,A_{g,t},B_{1,t},\dots,B_{g,t}\in H_1(\mathcal{X}_t,\mathbb{Z})$$
of $\mathcal{X}_t$, which continuously depends on $t\in D$. 
Furthermore, we can choose the deformations of $A_{r+1,t},\dots,A_{g,t},B_{r+1,t},\dots, B_{g,t}$, such that they give a continuous family of cycles on $U$ and do not depend on the choice of $D$. In general, we orient $A_{e,t}$ such that the intersection number $[A_{e,t},B_{j,t}]$ equals $b_{e,j}$.

To define a period matrix on $\mathcal{X}_t$ we choose a basis $\omega_{1,t},\dots,\omega_{g,t}\in H^0(\mathcal{X}_t,\Omega_{\mathcal{X}}^1)$ of $1$-forms for every $t\in U\setminus \{0\}$, such that $\int_{A_{j,t}}\omega_{i,t}=\delta_{ij}$. For any domain $D\subseteq U\setminus\{0\}$ we then define the period matrix by $\Omega(t)_{ij}=\int_{B_{j,t}}\omega_{i,t}$ for $t\in D$. This definition induces a multi-valued holomorphic map $\Omega\colon \Delta^*\to \mathbb{H}_g$. We would like to study its monodromy. Let $\gamma$ be a positive generator of $\pi_1(\Delta^*)$. As $\gamma$ only acts non-trivially on the cycles $B_{1,t},\dots, B_{r,t}$ it only changes the matrix $\Omega(t)$ in the upper left $r\times r$ square. The action of $\gamma$ can change the cycles $B_{1,t},\dots,B_{r,t}$ only by adding an integral linear combination of the cycles $A_{e,t}$ since $B_{i,t}$ and $\gamma(B_{i,t})$ have to deform to the same cycle on $\mathcal{X}_0$. As $A_{e,t}$ is locally given by $xy=t^{\ell(e)}$ and $|x|=|t|^{\ell(e)/2}$ for $e\in E(\Gamma_\pi)$, the action of $\gamma$ on $B_{j,t}$ is given by
$$B_{j,t}\xrightarrow{\gamma} B_{j,t}+\sum_{e\in E(\Gamma_\pi)}\ell(e) b_{j,e} A_{e,t}.$$
If we express $A_{e,t}$ in our symplectic basis of $H_1(\mathcal{X}_t,\mathbb{Z})$ only the cycles $A_{1,t},\dots,A_{r,t}$ can occur with non-trivial coefficient, since $A_{e,t}$ contracts to a point on $H_1(\mathcal{X}_0,\mathbb{Z})$. Further, the coefficient of $A_{j,t}$ in $A_{e,t}$ can be computed by the intersection number $[A_{e,t},B_j]$. By construction we have $[A_{e,t},B_j]=b_{e,j}$. Hence, we obtain
$$\sum_{e\in E(\Gamma_\pi)}\ell(e) b_{j,e} A_{e,t}\equiv\sum_{e\in E(\Gamma_\pi)}\ell(e) b_{j,e} b_{i,e} A_{i,t} \mod \mathrm{span}(A_{k,t}~|~k\neq i)$$
Thus, if we integrate with respect to the form $\omega_{i,t}$ we obtain that $\gamma$ acts on the entries of $\Omega(t)$ by
$$\Omega(t)_{ij}=\int_{B_{j,t}}\omega_{i,t}\xrightarrow{\gamma}\int_{B_{j,t}}\omega_{i,t}+\sum_{e\in E(\Gamma_\pi)}\ell(e) b_{j,e}b_{i,e}.$$
It follows, that $\Omega$ is of the form in (\ref{equ_period-matrix}) with $r=g_2$, $m=1$ and the entries of the matrix $B$ are given by $B_{ij}=\sum_{e\in E(\Gamma_\pi)}\ell(e)b_{j,e}b_{i,e}$. Note, that we placed the $g_2\times g_2$ matrix in the bottom right in (\ref{equ_period-matrix}) now in the top left for a better notation.

Finally, we would like to show that $\Sigma_f\cong\mathrm{\Gamma_\pi}$. It is enough to show that the map
$$\alpha\colon H_1(\Gamma_\pi,\mathbb{Z})\to B\mathbb{Z}^{g_2},\quad \sum_{i=1}^{g_2} c_i B_i\mapsto B\begin{pmatrix}c_1\\ \vdots\\ c_{g_2}\end{pmatrix}$$
is an isomorphism of Euclidean lattices. As it sends a basis onto a basis, it is an isomorphism of the underlying free $\mathbb{Z}$-modules. Further for $1\le i,j\le g_2$ it holds
\begin{align*}[B_i, B_j]_{\Gamma_\pi}=&\left[\sum_{e\in E(\Gamma_\pi)}b_{i,e}e,\sum_{e\in E(\Gamma_\pi)}b_{j,e}e\right]=\sum_{e\in E(\Gamma_\pi)}b_{i,e}b_{j,e}\ell(e)=B_{ij}\\
=&\ltrans{\alpha(B_i)}B^{-1}\alpha(B_j)=[\alpha(B_i),\alpha(B_j)]_B.
\end{align*}
Hence, $\alpha$ is an isomorphism of Euclidean lattices and the lemma is proved.
\end{proof}
Using this lemma, we can now prove Corollaries \ref{cor_ZK} and \ref{cor_lowerbound}.
\begin{proof}[Proof of Corollary \ref{cor_ZK}]
If we apply Lemma \ref{lem_period} to Theorem \ref{mainthm} for the family of Jacobians $f\colon \mathrm{Jac}(\mathcal{X}^*)\to \Delta^*$, we obtain
$$I(\mathrm{Jac}(\mathcal{X}_t))\sim -I(\mathrm{Jac}(\Gamma_{\pi}))\log|t|-\frac{\dim \Sigma_f}{2}\log(-\log|t|),$$
which can be combined with Equation (\ref{equ_delta-dejong}) to
$$\tfrac{1}{2}\delta(\mathcal{X}_t)-6I(\mathrm{Jac}(\mathcal{X}_t))\sim -\left(\tfrac{1}{2}\delta(\Gamma_\pi)+\tfrac{1}{2}\epsilon(\Gamma_\pi)-6\mathrm{Jac}(\Gamma_{\pi})\right)\log|t|.$$
By Equation (\ref{equ_delta}) we have $\varphi(\mathcal{X}_t)\sim \frac{1}{2}\delta(\mathcal{X}_t)-6I(\mathrm{Jac}(\mathcal{X}_t))$ and by Equation (\ref{equ_delta+epsilon}) it holds $\varphi(\Gamma_{\pi})=\tfrac{1}{2}\delta(\Gamma_\pi)+\tfrac{1}{2}\epsilon(\Gamma_\pi)-6\mathrm{Jac}(\Gamma_{\pi})$. Putting this into the above equation yields as desired
$\varphi(\mathcal{X}_t)\sim -\varphi(\Gamma_\pi)\log|t|$.
\end{proof}
\begin{proof}[Proof of Corollary \ref{cor_lowerbound}]
	Let $\mathcal{M}_g$ be the moduli space of compact and connected Riemann surfaces of genus $g\ge 2$ and $\overline{\mathcal{M}}_g$ its Deligne--Mumford compactification. We may consider $\varphi$ and $\delta$ as continuous functions on $\mathcal{M}_g$. We first would like to prove, that there is a positive constant $c_2'(g)$ such that $\varphi(M)\ge c_2'(g)$ for all $M\in\mathcal{M}_g$. As $\varphi(M)>0$ for every $M\in\mathcal{M}_g$ and  $\overline{\mathcal{M}}_g$ is compact, it is enough to show, that for every $p\in\overline{\mathcal{M}}_g\setminus\mathcal{M}_g$ and every embedding $\kappa\colon\Delta\to\overline{\mathcal{M}}_g$ with $\kappa(0)=p$ and $\kappa(\Delta^*)\subseteq\mathcal{M}_g$ we have $\lim_{t\to 0}\varphi(\kappa(t))\neq 0$. Let $\pi\colon \mathcal{X}\to \Delta$ be the pullback of the universal semistable curve along $\kappa$. Then the associated polarized metrized reduction graph $\Gamma_\pi$ is non-trivial and hence, $\varphi(\Gamma_\pi)>0$ by the inequality in (\ref{equ_bound-of-varphi}). Thus, Corollary \ref{cor_ZK} implies, that $\lim_{t\to 0}\varphi(\kappa(t))=\infty$ as desired.
	
	In a similar way we prove, that there is a constant $c_1(g)$ such that 
	$$f(M):=\tfrac{23g^2+11g+2}{(g-1)^2}\varphi(M)-\delta(M)\ge c_1(g)$$
	for all $M\in\mathcal{M}_g$.
	As $c_1(g)$ has not to be positive, it is enough to check, that for every $p\in\overline{\mathcal{M}}_g\setminus\mathcal{M}_g$ and every embedding $\kappa\colon\Delta\to\overline{\mathcal{M}}_g$ with $\kappa(0)=p$ and $\kappa(\Delta^*)\subseteq\mathcal{M}_g$ we have $\lim_{t\to 0}f(\kappa(t))>-\infty$. Again we denote $\pi\colon \mathcal{X}\to \Delta$ for the pullback of the universal semistable curve along $\kappa$ and $\Gamma_\pi$ for the associated polarized metrized reduction graph. We combine Corollary \ref{cor_ZK} and Equation (\ref{equ_delta-dejong-intro}) to obtain
	$$f(\kappa(t))\sim -\left(\tfrac{23g^2+11g+2}{(g-1)^2}\varphi(\Gamma_\pi)-\delta(\Gamma_\pi)-\epsilon(\Gamma_\pi)\right)\log|t| +6g_0(\Gamma_\pi)\log(-\log|t|).$$
	As $\lim_{t\to 0}\log(-\log|t|)=\infty$ and $\tfrac{23g^2+11g+2}{(g-1)^2}\varphi(\Gamma_\pi)- \delta(\Gamma_\pi)-\epsilon(\Gamma_\pi)\ge 0$ by the inequality in (\ref{equ_bound-of-varphi}), we obtain as desired $\lim_{t\to 0}f(\kappa(t))>-\infty$.
\end{proof}
\section{Arithmetic applications}\label{sec_arithmetic}
Finally in this section, we will apply our results to the arithmetic situation of a curve defined over a number field. In particular, we will give the proofs of Corollaries \ref{cor_height-bound} and \ref{cor_uniform-bogomolov}.

Let $X$ be a smooth projective geometrically connected curve of genus $g\ge 2$ over a number field $K$ of degree $d_K=[K\colon \mathbb{Q}]$ with semistable reduction over $B=\mathrm{Spec}~\mathcal{O}_K$. We denote $\hat{\omega}_X$ for the dualizing sheaf of $X$ equipped with its canonical admissible adelic metric defined in \cite{Zha93} and we write $\hat{\omega}_X^2$ for its arithmetic self-intersection number. Moreover, let $\pi\colon\mathcal{X}\to B$ be the minimal regular model of $X$ over $B$ and we denote $\omega_{\mathcal{X}/B}$ for the relative dualizing sheaf. The stable Faltings height is defined by $h_{\mathrm{Fal}}(X)=\frac{1}{d_K}\widehat{\deg}\det\pi_*\omega_{\mathcal{X}/B}$, where $\det\pi_*\omega_{\mathcal{X}/B}$ is equipped with its canonical metric induced by the inner product $\langle \omega,\omega'\rangle=\frac{i}{2}\int \omega\wedge \overline{\omega'}$ on $H^0(X_\mathbb{C},\Omega_{X_\mathbb{C}}^1)$ and $\widehat{\deg}$ denotes the Arakelov degree.

Let $v\in |B|$ be a closed point in $B$. Similarly as in Section \ref{sec_families_riemann}, we would like to define the polarized metrized reduction graph $\Gamma_v(X)$ of $X$ at $v$. The underlying graph of $\Gamma_v(X)$ is the dual graph of the special fiber $\mathcal{X}_v$ of $\pi$ at $v$. That means its vertex set $V$ consists of the irreducible components of $\mathcal{X}_v$ and its edge set consists of the singular points in $\mathcal{X}_v$, such that any edge connects the two not necessarily different vertices associated to the irreducible components meeting in the corresponding singular point. As $\mathcal{X}$ is regular, we assign to every edge $e$ of $\Gamma_v(X)$ the length $\ell(e)=1$ and we obtain a metrized graph, which we also call $\Gamma_v(X)$. We define a polarization on $\Gamma_v(X)$ by the effective divisor
$$K=\sum_{p\in V}(n_p+g_p-2)p,$$
where $n_p$ denotes the valency of the vertex $p$ as in Section \ref{sec_graphs} and $g_p$ denotes the genus of the normalization of the irreducible component of $\mathcal{X}_v$ corresponding to $p\in V$. The polarized metrized graph $\Gamma_v(X)=(\Gamma_v(X),K)$ has genus $g$.

We write for shorter notations
\begin{align*}
\delta(X)=&\sum_{v\in|B|}(\delta(\Gamma_v(X))+\epsilon(\Gamma_v(X)))\log N(v)+\sum_{v\colon K\to\mathbb{C}}(\delta(X_v)-4g\log 2\pi),\\
\varphi(X)=&\sum_{v\in|B|}\varphi(\Gamma_v(X))\log N(v)+\sum_{v\colon K\to\mathbb{C}}\varphi(X_v)
\end{align*}
where $N(v)=\#\mathcal{O}_K/v\mathcal{O}_K$, in each line the second sum runs over all embeddings $v\colon K\to \mathbb{C}$ and $X_v$ is the Riemann surface obtained by the pullback of $X$ induced by $v$.
In particular, we obtain from (\ref{equ_bound-of-varphi}) and from Corollary \ref{cor_lowerbound} that
\begin{align}\label{inequ_varphidelta}
\tfrac{23g^2+11g+2}{(g-1)^2}\varphi(X)\ge \max(\delta(X)+d_Kc_1(g),d_Kc_2(g))
\end{align}
with $c_1(g)$ and $c_2(g)$ as in Corollary \ref{cor_lowerbound}.

In this notation the arithmetic Noether formula proven by Faltings \cite[Theorem~6]{Fal84} states
\begin{align}\label{equ_noether}
12d_K h_{\mathrm{Fal}}(X)=\hat{\omega}_X^2+\delta(X),
\end{align}
where we refer to \cite{Mor89} for the constant term and to \cite[Theorem 5.5]{Zha93} for the difference induced by using the admissible metrized bundle $\hat{\omega}_X$. Further, we recall from \cite[Theorem 1.2]{Wil19} the inequality
\begin{align}\label{inequ_omegavarphi}
\tfrac{2g+1}{g-1}\hat{\omega}_X^2\ge \varphi(X).
\end{align}
With these bounds we are able to prove Corollary \ref{cor_height-bound}.
\begin{proof}[Proof of Corollary \ref{cor_height-bound}]
	We compute
	\begin{align*}
		\tfrac{47g^3+42g^2+18g+1}{(g-1)^3}\hat{\omega}_X^2&=\hat{\omega}_X^2+\tfrac{23g^2+11g+2}{(g-1)^2}\cdot\tfrac{2g+1}{g-1}\hat{\omega}_X^2\ge \hat{\omega}_X^2+\tfrac{23g^2+11g+2}{(g-1)^2}\varphi(X)\\
		&\ge \max(\hat{\omega}_X^2+\delta(X)+d_Kc_1(g),\hat{\omega}_X^2+d_Kc_2(g))\\
		&\ge d_K\max(12h_{\mathrm{Fal}}(X)+c_1(g), c_2(g)),
	\end{align*}
	where the first inequality follows by (\ref{inequ_omegavarphi}), the second one follows from (\ref{inequ_varphidelta}) and the third one follows by the arithmetic Noether formula (\ref{equ_noether}) and the non-negativity of $\hat{\omega}_X^2$. Now the corollary follows by dividing by the coefficient on the left hand side.
\end{proof}
Next we would like to prove Corollary \ref{cor_uniform-bogomolov}. Let $J=\mathrm{Pic}^0(X)$ be the Jacobian variety of $X$ and $\mathcal{L}$ a symmetric ample line bundle on $J$ defining the canonical principal polarization on $J$ and which is rigidified at the origin. Then the multiplication-by-2-map $[2]\colon J\to J$ induces a unique isomorphism $\varphi\colon \mathcal{L}^{\otimes 4}\cong [2]^*\mathcal{L}$. Due to Zhang \cite[Theorem 2.2]{Zha95} there is a unique adelic metric on $\mathcal{L}$ such that this isomorphism becomes an isometry. We write $\hat{\mathcal{L}}$ for the corresponding adelic metrized line bundle.

The Néron--Tate height of an integral cycle $Z\subseteq J$ of dimension $\dim Z=r$ is defined as
$$h_{\mathcal{L}}'(Z)=\frac{\langle\hat{\mathcal{L}}^{r+1}|Z\rangle}{d_K(r+1)\langle\mathcal{L}^r|Z\rangle},$$
where the terms in the brackets denote the (arithmetic) self-intersection number of $\mathcal{L}$ (respectively $\hat{\mathcal{L}}$) restricted to $Z$.
If $1\le r\le g-1$ is an integer, $m\in(\mathbb{Z}\setminus\{0\})^r$ is any vector of non-zero integers and $\alpha\in\mathrm{Div}^1(X)$ is any divisor of degree $1$ on $X$, we define the map $f_{m,\alpha}$ by
$$f_{m,\alpha}\colon X^r\to J,\quad (x_1,\dots, x_r)\mapsto \sum_{j=1}^rm_j(x_j-\alpha)$$
and we write $Z_{m,\alpha}$ for the cycle in $J$ obtained by the image of $f_{m,\alpha}$. In \cite[Theorem~1.1]{Wil19} we computed $h'_{\mathcal{L}}(Z_{m,\alpha})$ and we may deduce from this result
\begin{align}\label{inequ_zmalpha}
h'_{\mathcal{L}}(Z_{m,\alpha})\ge \tfrac{g-r}{2d_K}\left(\left(\tfrac{\sum_{j=1}^r m_j^2}{4(g-1)^2}-\tfrac{(2g+1)\sum_{j<k}^r m_j m_k}{6g(g-1)^2(g-2)}\right)\hat{\omega}_X^2+\tfrac{\sum_{j<k}^r m_j m_k}{3g(g-1)(g-2)}\varphi(X)\right),
\end{align}
where for $r=1$ the fractions with factor $\sum_{j<k}^rm_jm_k=0$ in the nominator are also set to be $0$ if $g=2$. 
This inequality allows us to prove Corollary \ref{cor_uniform-bogomolov}.
\begin{proof}[Proof of Corollary \ref{cor_uniform-bogomolov}]
	First let us bound the right hand side of the inequality (\ref{inequ_zmalpha}) independently of $m$. We will use the following elementary estimates
	\begin{align}\label{inequ_m}
	\sum_{j<k}^r m_j m_k&=\tfrac{r-1}{2}\sum_{j=1}^r m_j^2-\tfrac{1}{2}\sum_{j<k}^r(m_j-m_k)^2\le \tfrac{r-1}{2}\sum_{j=1}^r m_j^2,\\
	\sum_{j<k}^r m_j m_k&=\tfrac{1}{2}\left(\sum_{j=1}^r m_j\right)^2-\tfrac{1}{2}\sum_{j=1}^r m_j^2\ge -\tfrac{1}{2}\sum_{j=1}^r m_j^2.\label{inequ_m2}
	\end{align}
    There are two cases to distinguish.
	\begin{enumerate}[(i)]
		\item Assume that $\sum_{j<k}^rm_jm_k\ge 0$. Then it holds
		\begin{align*}
		&\left(\tfrac{\sum_{j=1}^r m_j^2}{4(g-1)^2}-\tfrac{(2g+1)\sum_{j<k}^r m_j m_k}{6g(g-1)^2(g-2)}\right)\hat{\omega}_X^2+\tfrac{\sum_{j<k}^r m_j m_k}{3g(g-1)(g-2)}\varphi(X)\\
		&\ge\tfrac{(3g(g-2)-(2g+1)(r-1))\sum_{j=1}^rm_j^2}{12g(g-1)^2(g-2)}\hat{\omega}_X^2\ge\tfrac{r}{12g(g-1)}\hat{\omega}_X^2,
		\end{align*}
		where we just used the assumption and (\ref{inequ_m}) for the first inequality and in the second one we used $r\le g-1$ and $\sum_{j=1}^r m_j^2\ge r$.
		\item
		Assume that $\sum_{j<k}^rm_jm_k<0$. In particular, we have $r\ge 2$ and hence, $g\ge 3$. Then it holds
		\begin{align*}
		&\left(\tfrac{\sum_{j=1}^r m_j^2}{4(g-1)^2}-\tfrac{(2g+1)\sum_{j<k}^r m_j m_k}{6g(g-1)^2(g-2)}\right)\hat{\omega}_X^2+\tfrac{\sum_{j<k}^r m_j m_k}{3g(g-1)(g-2)}\varphi(X)\\
		&\ge\left(\tfrac{\sum_{j=1}^r m_j^2}{4(g-1)^2}+\tfrac{(2g+1)\sum_{j<k}^r m_j m_k}{6g(g-1)^2(g-2)}\right)\hat{\omega}_X^2\ge\tfrac{(3g(g-2)-(2g+1))\sum_{j=1}^rm_j^2}{12g(g-1)^2(g-2)}\hat{\omega}_X^2\ge\tfrac{r}{12g(g-1)}\hat{\omega}_X^2
		\end{align*}
		where we applied (\ref{inequ_omegavarphi}) in the first inequality and we used (\ref{inequ_m2}) in the second one. The third inequality follows as
		$3g(g-2)-(2g+1)\ge (g-2)(g-1)$ for $g\ge 3$ and $\sum_{j=1}^r m_j^2\ge r$.
	\end{enumerate}
	Applying this to (\ref{inequ_zmalpha}) we can bound $h_{\mathcal{L}}'(Z_{m,\alpha})$ by
	$$h_{\mathcal{L}}'(Z_{m,\alpha})\ge \tfrac{(g-r)r}{24d_Kg(g-1)}\hat{\omega}_X^2\ge \tfrac{1}{24d_Kg}\hat{\omega}_X^2.$$
	If we now apply Corollary \ref{cor_height-bound}, we obtain the lower bound
	$$h_{\mathcal{L}}'(Z_{m,\alpha})\ge\tfrac{(g-1)^3}{24(47g^4+42g^3+18g^2+g)}\max(12h_{\mathrm{Fal}}(X)+c_1(g),c_2(g))$$
	as claimed in the corollary.
\end{proof}

\vspace{0.5cm}
Robert Wilms \\
Department of Mathematics and Computer Science\\
University of Basel\\
Spiegelgasse 1\\
4051 Basel\\
Switzerland\\
E-mail: robert.wilms@unibas.ch
\end{document}